\documentclass[11pt,twoside]{article}

\usepackage{mathrsfs,amsfonts,amsmath,amssymb}
\usepackage{latexsym}
\usepackage{comment}
\newtheorem{satz}{Theorem}
\newtheorem{proposition}[satz]{Proposition}
\newtheorem{theorem}[satz]{Theorem}
\newtheorem{lemma}[satz]{Lemma}
\newtheorem{definition}[satz]{Definition}
\newtheorem{corollary}[satz]{Corollary}
\newtheorem{remark}[satz]{Remark}

\def\T{\mathsf{T}}

\def\Z{\mathbb {Z}}
\def\F{\mathbb {F}}
\def\E{\mathsf{E}}

\def\a{\alpha}

\def\C{\mathbb{C}}

\def\d{\delta}
\def\o{\omega}
\def\({\big (}
\def\){\big )}

\def\G{\Gamma}

\def\dim{{\rm dim}}
\def\le{\leqslant}
\def\ge{\geqslant}
\def\_phi{\varphi}
\def\eps{\varepsilon}

\def\Gr{{\mathbf G}}
\def\FF{\widehat}

\def\ov{\overline}

\def\D{\Delta}

\def\Bohr{{\rm Bohr}}
\def\Stab{{\rm Stab}}

\def\tr{\mathrm{tr}}
\def\Bohr{{\rm Bohr}}

\author{Shkredov I.D.}
\title{On multiplicative energy of subsets of varieties
	\footnote{This work is supported by the Russian Science Foundation under grant 19--11--00001.}
}
\date{}
\begin{document}
\maketitle

\begin{center}
	Annotation.
\end{center}

{\it \small
	We obtain a non--trivial upper bound for the multiplicative energy of any sufficiently large subset of a subvariety of a finite algebraic group.      
	We also find some applications of our results to growth of conjugates classes, estimates of exponential sums and the restriction phenomenon.
}
\\

\section{Introduction}

In papers \cite{BGS}, \cite{BGT},  \cite{Gow_random}, \cite{H}, \cite{H_ideas},  \cite{LP}, \cite{LS1}, \cite{LSS} and in many others authors study growth properties of   rather general subsets $A$ of different groups  $\Gr$ (basically, of Lie type). 
One of the difficulties concerning growth of $A$ is that, in principle, $A$ can live in a subvariety of $\Gr$, see   \cite{H}, \cite{H_ideas},  \cite{LP}. 
In this article we restrict ourselves to the case when $A$ indeed belongs to a subvariety and consider the most natural combinatorial problem  connecting growth of $A$, namely, the basic question about obtaining upper bounds for  the {\it multiplicative energy} (see, e.g., \cite{TV}) of $A$ 
$$
	\E(A) := |\{ (a,b,c,d) \in A^4 ~:~ ab^{-1} = c d^{-1} \}| \,. 
$$
Our result is the following 

\begin{theorem}
	Let  $\Gr$ be  a finite  algebraic group over $\F_q$,  $V\subseteq \Gr$ be a  variety and $\G$ be a maximal algebraic subgroup such that a coset of $\G$ is contained in $V$.
	Then for any $A\subseteq V$, $|A|\ge |\G|^{1+\eps}$
	and all sufficiently large $q$ 
	one has 
\begin{equation}\label{f:A_op_intr}
	\E(A) \ll |A|^{3-\delta} \,,
\end{equation}
	where $\d = \d(\eps, \dim (V)) >0$ and the implied constant in \eqref{f:A_op_intr} depends on $\dim(V), \deg(V)$,  
	$\dim(\Gr), \deg(\Gr)$ and dimension $n$ of $\F_q^n \supseteq \Gr$.\\ 
	In particular, bound \eqref{f:A_op_intr} takes place for a variety $V$ iff $V$ does not contain a coset of an algebraic subgroup of size $\Omega (|V|)$.  
\label{t:A_op_intr}
\end{theorem}

Theorem above gives us a non--trivial upper bound for the operator norm of $\FF{A}$ (all definitions can be found in section \ref{sec:def}) for  any sufficiently large subset of a  Chevalley group living in a variety differ from the maximal parabolic subgroup.

\begin{theorem}
	Let $\Gr_r (\F_q)$ be a finite Chevalley group with rank $r$ and odd $q$ and let $\Pi \le \Gr_r (\F_q)$ be its maximal (by size) parabolic subgroup.
	Also, let $V \subset \Gr_r (\F_q)$ be a variety differ from all shifts of conjugates of  $\Pi$.
	Then for any $A\subseteq V$, $|A| \ge |\Pi| q^{-1+c}$, $c>0$ one has 
	\begin{equation}\label{f:A_op_intr'}
	\| \FF{A} (\rho) \|_o \le |A|^{1-\d} \,,
	\end{equation}
	where $\d = \d(c,r) >0$ and $\rho$ is any non--trivial unitary representation of $\Gr_r (\F_q)$. 
	\label{t:Chevalley_A_intr}
\end{theorem} 

Bound \eqref{f:A_op_intr'} implies the uniform distribution of $A$ 
among 
any sets with small product, see Proposition \ref{p:UD_non-commutative} below.
It is interesting that all  our conditions in Theorems \ref{t:A_op_intr}, \ref{t:Chevalley_A_intr}  concerning intersection of $A$ with subgroups  are formulated in terms of $V$ but not $A$.
In a similar way, we do not require that $A$ is a generating set of $\Gr$. 
It differs our result from Larsen--Pink machinery, see \cite{LP} and also \cite{BGT}.


Theorem \ref{t:A_op_intr} has a naturally--looking algebraic consequence (see rigorous formulation in section \ref{sec:proof}).

\begin{corollary}
	Suppose that  $\Gr$ is  a finite  algebraic group, and  $V\subseteq \Gr$ is   
a variety.
	Then size of maximal (by cardinality)  subgroup of $V$ is comparable with size of maximal (by cardinality) algebraic subgroup of $V$. 
\label{c:A_op_intr}	
\end{corollary}

In other words, 
if $\Gamma_a$ is a maximal by cardinality algebraic subgroup in shifts of $V$, then for any $x$ and $\G \le \Gr$ such that $x\G \subseteq V$ one has $|\G| = O(|\G_a|)$.
In big--O here we assume that $q \to \infty$ and the implied constant depends on $\dim(V), \deg(V), \dim(\Gr), \deg(\Gr)$ and dimension of the ground affine space.

We obtain several applications of 
Theorem \ref{t:A_op_intr}. 
In the first one we take our variety $V$ be a (Zariski closure of) conjugate class $C$ of a finite algebraic group.
In \cite{LS1}, \cite{LSS} authors obtain that for any such $C$ one has $|CC| \gg \min \{ |C|^{2-o(1)}, |\Gr| \}$. 
We prove that in certain cases one has $|AA| \gg |A|^{1+c}$, $c>0$ for {\it any} sufficiently large subset $A$ of $C$.

Of course Theorem \ref{t:Chevalley_A_intr}
is based on a purely non--commutative phenomenon of growth in groups and, say, estimate \eqref{f:A_op_intr'} does not hold in $\F^n_q$. 
Nevertheless,  
in section \ref{sec:applications} we obtain a purely commutative application to so--called {\it restriction} problems.
Here we have deal with the restriction phenomenon in finite fields, see \cite{MoT} and good survey \cite{IKL}. 
We prove (all definitions can be found in section \ref{sec:appendix})

\begin{theorem}
	Let $V\subseteq \F^n_q$ be a variety, $d=\dim(V)$.
	Suppose that $V$ does not contain any line. 
	Then $R^* (\frac{4}{3-c} \rightarrow 4) \lesssim 1$, where $c=c(d) > 0$. 
	\label{t:restriction_intr}
\end{theorem}

In papers 
\cite{IK}---\cite{MoT} and \cite{Volobuev} 
authors consider some particular varieties as cones, paraboloids and  spheres. 
Our result is weaker but on the other hand we have deal with an almost arbitrary 
variety $V$.


We thank Nikolai Vavilov for  useful discussions.  
We deeply thank Brendan Murphy for his idea to study energies of subsets of various varieties.  

\section{Definitions}
\label{sec:def}

Let $\Gr$ be a group with the identity $1$.
Given two sets $A,B\subset \Gr$, define  the \textit{product set}  of $A$ and $B$ as 
$$AB:=\{ab ~:~ a\in{A},\,b\in{B}\}\,.$$
In a similar way we define the higher product sets, e.g., $A^3$ is $AAA$. 
Let $A^{-1} := \{a^{-1} ~:~ a\in A \}$. 
As usual, having two subsets $A,B$ of a group $\Gr$,  denote by 
\[
\E(A,B) = |\{ (a,a_1,b,b_1) \in A^2 \times B^2 ~:~ a^{-1} b = a^{-1}_1 b_1 \}| 
\]
the {\it common energy} of $A$ and $B$. 
Clearly, $\E(A,B) = \E(B,A)$ and by the Cauchy--Schwarz inequality 
\begin{equation}\label{f:energy_CS}
\E(A,B) |A^{-1} B| \ge |A|^2 |B|^2 \,.
\end{equation}
In a little more general way define
\[
	\E^L_k (A) =  |\{ (a_1,\dots, a_k, b_1, \dots, b_k) \in A^{2k}  ~:~ a^{-1}_1 b_1  = \dots = a^{-1}_k b_k  \}| \,,
\]
and, similar, one can define $\E^R_k (A)$. 
For $k=2$, we have $\E^L_k (A) = \E^R_k (A)$ but for larger $k$ it is not 
the case. 
If there is no difference between  $\E^L_k (A)$ and $\E^R_k (A)$, then we write just $\E_k (A)$.  
In this paper we use the same letter to denote a set $A\subseteq \Gr$ and  its characteristic function $A: \Gr \to \{0,1 \}$.

First of all, we recall some notions and simple facts from the representation theory, see, e.g., \cite{Naimark} or \cite{Serr_representations}.
For a finite group $\Gr$ let $\FF{\Gr}$ be the set of all irreducible unitary representations of $\Gr$. 
It is well--known that size of $\FF{\Gr}$ coincides with  the number of all conjugate classes of $\Gr$.  
For $\rho \in \FF{\Gr}$ denote by $d_\rho$ the dimension of this representation. 
Thus $\Gr$ is a quasi--random group in the sense of Gowers (see \cite{Gow_random}) iff 
$d_{\rho} \ge |\Gr|^\eps$, where $\eps>0$ and $\rho$ is any non--trivial irreducible unitary representation of $\Gr$.  
We write $\langle \cdot, \cdot \rangle$ for the corresponding  Hilbert--Schmidt scalar product 
$\langle A, B \rangle = \langle A, B \rangle_{HS}:= \tr (AB^*)$, where $A,B$ are any two matrices of the same sizes. 
Put $\| A\| = \sqrt{\langle A, A \rangle}$.
Finally, it is easy to check that for any matrices $A,B$ one has $\| AB\| \le \| A\|_{o} \| B\|$ and $\| A\|_{o} \le \| A \|$, where  the operator $l^2$--norm  $\| A\|_{o}$ is just 
the maximal singular value of $A$.  

For any function $f:\Gr \to \mathbb{C}$ and $\rho \in \FF{\Gr}$ define the matrix $\FF{f} (\rho)$, which is called the Fourier transform of $f$ at $\rho$ by the formula 
\begin{equation}\label{f:Fourier_representations}
\FF{f} (\rho) = \sum_{g\in \Gr} f(g) \rho (g) \,.
\end{equation}
Then the inverse formula takes place
\begin{equation}\label{f:inverse_representations}
f(g) = \frac{1}{|\Gr|} \sum_{\rho \in \FF{\Gr}} d_\rho \langle \FF{f} (\rho), \rho (g^{-1}) \rangle \,,
\end{equation}
and the Parseval identity is 
\begin{equation}\label{f:Parseval_representations}
\sum_{g\in \Gr} |f(g)|^2 = \frac{1}{|\Gr|} \sum_{\rho \in \FF{\Gr}} d_\rho \| \FF{f} (\rho) \|^2 \,.
\end{equation}
The main property of the Fourier transform is the convolution formula 
\begin{equation}\label{f:convolution_representations}
\FF{f*g} (\rho) = \FF{f} (\rho) \FF{g} (\rho) \,,
\end{equation}
where the convolution 
of two functions $f,g : \Gr \to \mathbb{C}$ is defined as 
\[
(f*g) (x) = \sum_{y\in \Gr} f(y) g(y^{-1}x) \,.
\]
Given a function $f : \Gr \to \mathbb{C}$ and a positive integer $k$, we write  $f^{(k)} = f^{(k-1)} * f$ for the $k$th convolution of $f$.  
Now let $k\ge 2$ be an integer and $f_j : \Gr \to \mathbb{C}$, $j\in [2k]$ be any functions. 
Denote by $\mathcal{C}$ the operator of convex conjugation. 
As in \cite{s_Kloosterman} define
\begin{equation}\label{def:T_k}
	\T_k (f_1,\dots, f_{2k}) = 
		\frac{1}{|\Gr|} \sum_{\rho \in \FF{\Gr}} d_\rho \langle \prod_{j=1}^k \mathcal{C}^j \FF{f}_j (\rho), \prod_{j=k+1}^{2k} \mathcal{C}^j \FF{f}_j (\rho)  \rangle \,.
\end{equation}
Put $\T_k (f) = \T_k (f,\dots, f_{})$.
For example,  we have, clearly, $\T_2 (A) = \E (A)$. 
It is easy to see that $\T^{1/2k}_k (f)$ defines a norm of a function $f$ (see \cite{s_Kloosterman}). 
This fact follows from the following inequality \cite[Lemma 10]{s_Kloosterman}
\begin{equation}\label{f:T_k_product}
	\T^{2k}_k (f_1,\dots, f_{2k}) \le \prod_{j=1}^{2k} \T_k (f_j) \,.
\end{equation}
In particular, $\E(A,A^{-1}) \le \E(A)$.

Now let us say a few words about varieties. 
Having a field $\F$ define an (affine) variety in $\F^n$ 
to be  
the  set of the form 
\[
	V = \{ (x_1,\dots, x_n) \in \F^n ~:~ p_j (x_1,\dots,x_n) = 0\, \mbox{ for all } j \} \,, 
\]
where $p_j \in \F[x_1,\dots,x_n]$. 
Let us recall some basic properties of varieties.
General theory of varieties and schemes  
can be found,
e.g., in  \cite{Hartshorne}. 
The union of any finite number of varieties is, clearly, a variety and the intersection of any number of varieties is a variety as well. 
Having a set $X$ we denote by $\mathrm{Zcl} (X)$ a minimal  (by inclusion) variety, containing $X$.
A variety is {\it irreducible} if it is not the union of two proper subvarieties. 
Every variety has a unique (up to inclusion) decomposition into finitely many irreducible components \cite{Hartshorne}. 
The {\it dimension} of $V$ is 
\[
	\dim(V) = \max \{ n ~:~ V\supseteq X_n \supset X_{n-1} \supset \dots \supset X_0 \neq \emptyset \} \,,	
\]
where $X_j$ are irreducible subvarieties of $V$.  
We will frequently use the simple fact that if $V_1 \subseteq V_2$ are two varieties and $V_2$ is irreducible, then either $V_1=V_2$, or $\dim(V_1) < \dim (V_2)$. 
A variety is {\it absolutely irreducible} if it is irreducible over $\ov{\F}$. 
In this paper we consider just these varieties.

We define the {\it degree} of any irreducible variety $V$ with $\dim (V) = d$ as in \cite{Heintz}, namely,
\[
	\deg (V) = \sup \{ |L \cap V| <\infty ~:~ L \mbox{ is } (n-d)\mbox{-dimensional affine subspace in }  \F^n \} \,.
\]
For an arbitrary variety $V$ we denote $\deg(V)$ to be the sum of the degrees of its irreducible components.
Recall generalized B\'ezout Theorem (see \cite[Theorem 1]{Heintz}) : for any varieties $U,V$ one has 
\begin{equation}\label{f:deg_intersection}
\deg (U\cap V) \le \deg(U) \deg(V) \,.
\end{equation}

The signs $\ll$ and $\gg$ are the usual Vinogradov symbols.
If we want to underline the dependence on a parameter $M$, then we write $\ll_M$ and $\gg_M$. 
All logarithms are to base $2$.
Sometimes we allow ourselves to lose logarithmic powers of $|\F|$. 
In this situation we write $\lesssim$ and $\gtrsim$ instead of $\ll$, $\gg$.

\section{On non--commutative Gowers norms}

Let  $\Gr$ be a group  and  $A\subseteq \Gr$ be a finite set.
Let 
$
\| A \|_{\mathcal{U}^{k}}
$
be the Gowers non--normalized $k$th--norm \cite{Gow_m} of the characteristic function of $A$ (in multiplicative form), see, say \cite{s_energy}:
$$
\| A \|_{\mathcal{U}^{k}}
=
\sum_{x_0,x_1,\dots,x_k \in \Gr}\,  \prod_{\vec{\varepsilon} \in \{ 0,1 \}^k} A \left( x_0 x^{\varepsilon_1}_1 \dots x^{\varepsilon_k}_k \right) \,,
$$
where $\vec{\varepsilon} = (\varepsilon_1,\dots,\varepsilon_k)$. 
For example,
$$
\| A \|_{\mathcal{U}^{2}} = \sum_{x_0,x_1,x_2\in \Gr} A(x_0) A(x_0 x_1 ) A(x_0 x_2) A(x_0 x_1 x_2) = \E (A)
$$
is the  energy of $A$ and $\| A \|_{\mathcal{U}^{1}} = |A|^2$. 
For any $\vec{s} = (s_1,\dots,s_k)  \in \Gr^k$ put 
\begin{equation}\label{def:A_vec{s}}
	A_{\vec{s}} (x) = \prod_{\vec{\varepsilon} \in \{ 0,1 \}^k} A \left(x s^{\varepsilon_1}_1 \dots s^{\varepsilon_k}_k \right) \,,
\end{equation}
and similar for an arbitrary function $f: \Gr \to \C$, namely, 
$$f_{\vec{s}} (x) = \prod_{\vec{\varepsilon} \in \{ 0,1 \}^k} \mathcal{C}^{\eps_1+\dots+\eps_k} f \left(x s^{\varepsilon_1}_1 \dots s^{\varepsilon_k}_k \right) \,,$$ 
where $\mathcal{C}$ is the operator of the conjugation.
E.g., $A_s (x) = A(x) A(xs)$ or, in other words,  $A_s = A\cap (As^{-1})$. 
Then, obviously, 
\begin{equation}\label{f:Gowers_sums_A_s}
	\| A \|_{\mathcal{U}^{k}} = \sum_{\vec{s}} |A_{\vec{s}}| \,.
\end{equation}
Also note that
\begin{equation}\label{f:Gowers_sq_A}
\| A \|_{\mathcal{U}^{k+1}} = \sum_{\vec{s}} |A_{\vec{s}}|^2 \,.
\end{equation}
Moreover, the induction property for Gowers norms holds (it follows from the definitions or see \cite{Gow_m}) 
\begin{equation}\label{f:Gowers_ind}
	\| A \|_{\mathcal{U}^{k+1}} = \sum_{s \in  A^{-1}A} \| A_s \|_{\mathcal{U}^{k}} \,,
\end{equation}
e.g., in particular, 
$$
	\| A \|_{\mathcal{U}^{3}} = \sum_{s \in A^{-1}A} \E(A_s) \,.
$$
The Gowers norms enjoy the following weak commutativity property. 
Namely, let $k=n+m$, and $\vec{s} = (s_1,\dots,s_k) = (\vec{u}, \vec{v})$, where the vectors $\vec{u}, \vec{v}$ have lengths $n$ and $m$, correspondingly. 
We have 
\[
	\| A \|_{\mathcal{U}^k} = \sum_{\vec{s}} \sum_x \prod_{\vec{\varepsilon} \in \{ 0,1 \}^k} A \left( x s^{\varepsilon_1}_1 \dots s^{\varepsilon_k}_k  \right)
	=
\]
\begin{equation}\label{f:comm_Uk}
	=
	\sum_{\vec{u},\vec{v}}\, \sum_{x} \prod_{\vec{\eta} \in \{ 0,1 \}^m}\, \prod_{\vec{\o} \in \{ 0,1 \}^n} 
		A \left( u^{\eta_1}_1 \dots u^{\eta_m}_m x v^{\o_1}_1 \dots v^{\o_n}_n  \right) \,.
\end{equation}
In particular, $\| A^{-1} \|_{\mathcal{U}^k} = 	\| A \|_{\mathcal{U}^k}$ and  $\| gA \|_{\mathcal{U}^k} = \| Ag \|_{\mathcal{U}^k} = \| A \|_{\mathcal{U}^k}$ for any $g\in \Gr$. 
To 
obtain 
\eqref{f:comm_Uk} just make the changing of variables $u_j x = x \tilde{u}_j$ for $j\in [m]$.

It was proved in \cite{Gow_m} that ordinary Gowers $k$th--norms of the characteristic function of any subset of an abelian group $\Gr$ are connected to each other.
In \cite{s_energy} the  author shows that the connection for the non--normalized norms does not depend on size of the group $\Gr$. 
Here we formulate a particular case of Proposition 35 from \cite{s_energy}, which relates  $\| A \|_{\mathcal{U}^{k}}$ and $\| A \|_{\mathcal{U}^{2}}$, see Remark 36 here.

\begin{lemma}
	Let $A$ be a finite subset of a commutative group $\Gr$.
	Then for any integer $k\ge 1$ one has
	$$
	\| A \|_{\mathcal{U}^{k+1}} \ge
	\frac{\| A \|^{(3k-2)/(k-1)}_{\mathcal{U}^{k}}}{\| A \|^{2k/(k-1)}_{\mathcal{U}^{k-1}}} \,.
	$$
	In particular, 
	$$
	\| A \|_{\mathcal{U}^{k}} \ge \E(A)^{2^k-k-1} |A|^{-(3\cdot 2^k -4k -4)} \,.
	$$
	\label{l:Gowers_char}
\end{lemma}

Actually, one can derive Lemma \ref{l:Gowers_char} 
from Lemma \ref{l:Gowers_char_AB} below but to prove this more general result we need an additional notation and arguments.

Given two functions $f,g: \Gr \to \mathbb{C}$ and an integer $k\ge 0$ consider the "scalar product"
\[
	\langle f, g \rangle_k := \sum_{\vec{s}, t} \sum_{x} f_{\vec{s}} (x) \overline{g_{\vec{s}} (xt)} = \overline{\langle g, f \rangle_k} \,, 
\]
where $\vec{s} = (s_1,\dots, s_k)$. 
For example, $\langle A, B \rangle_1 = \E(A,B)$, $\langle A, B \rangle_0 = |A||B|$, $\langle A, A \rangle_k = \|A\|_{\mathcal{U}^{k+1}}$,  
$\langle A, 1 \rangle_k = |\Gr| \|A\|_{\mathcal{U}^{k}}$ (for finite group $\Gr$). 
Clearly, $\langle f, g \rangle_0 = \left( \sum_x f(x) \right) \left( \overline{\sum_x g(x)} \right)$ but for $k\ge 1$ it is easy to see that 
$\langle f, g \rangle_k \ge 0$ because 
$\langle f, g \rangle_k = 
\sum_{\vec{s}_*, s_k} |(f_{\vec{s}_*} * \ov{\tilde{g}}_{\vec{s}_*}) (s_k)|^2 \ge 0$, 
where $\vec{s}_* = (s_1,\dots, s_{k-1})$, and $\tilde{g}(x) := g(x^{-1})$. 
Also note that
\begin{equation}\label{f:A,B_scalar_exm}	
	\langle A, B \rangle_k = \sum_{\vec{s}} |A_{\vec{s}}| |B_{\vec{s}}| 
		= 
			\sum_{\vec{s}_*}  \sum_{s_k} |A_{\vec{s}_*} \cap A_{\vec{s}_*} s_k| |B_{\vec{s}_*} \cap B_{\vec{s}_*} s_k| \,.
\end{equation}

\begin{lemma}
	Let $\Gr$ be a commutative group and 
	$f,g : \Gr \to \mathbb{C}$ be functions.
	Then for any integer $k\ge 1$ one has
\begin{equation}\label{f:Gowers_char_AB_1}
	\langle f, g\rangle^{3+1/k}_k \le \langle f, g\rangle^{2}_{k-1} \langle f, g\rangle^{}_{k+1} \| f\|^{1/k}_{\mathcal{U}^k} \| g\|^{1/k}_{\mathcal{U}^k} \,,
\end{equation}
	and hence for an arbitrary $k\ge 2$ and any sets $A,B \subseteq \Gr$  the following holds 
\begin{equation}\label{f:Gowers_char_AB_2}
	\E(A,B) \le (|A||B|)^{\frac{3}{2} - \frac{\beta(k+2)}{2}}  \langle A,B \rangle^{\beta}_k  \,,
\end{equation}	
	where $\beta = \beta(k) \in [4^{-k}, 2^{-k+1}]$.\\
	For any (not necessary commutative) group $\Gr$ if $\|A\|_{\mathcal{U}^k} \le |A|^{k+1-c}$, where $c>0$, then $\E(A) \le |A|^{3-c_*}$ with $c_* = c_*(c,k)>0$. 
\label{l:Gowers_char_AB}
\end{lemma}
\begin{proof}
	We have 
\begin{equation}\label{tmp:17.12_1}
	\sigma:= \langle f, g \rangle_k = \sum_{\vec{s}, t} \sum_{x} f_{\vec{s}} (x) \overline{g_{\vec{s}} (xt)}
\end{equation}
	and our first task to estimate size of the set of $(\vec{s},t)$ in the last formula. 
	Basically, we consider two cases.
	If the summation in (\ref{tmp:17.12_1}) is taken over the set
\begin{equation}\label{tmp:Q_def}
	Q := \{ \vec{s} ~:~ \sum_t g_{\vec{s}} (t) \ge \sigma (2k \| f \|_{\mathcal{U}^{k}} )^{-1} \} \,,
\end{equation}
then it gives us $(1-1/2k)$ proportion of $\sigma$.
	Cardinality of the set $Q$ can be estimated as 
\[
	|Q| \cdot \sigma (2k \| f \|_{\mathcal{U}^{k}} )^{-1} \le  \|g \|_{\mathcal{U}^{k}}
\]	
and hence $|Q| \le 2k \| f \|_{\mathcal{U}^{k}} \| g \|_{\mathcal{U}^{k}} \sigma^{-1}$. 
	Now we fix any $j\in [k]$ (without any loss of generality we can assume that $j=1$) put $\vec{s}_* = (s_2,\dots,s_{k})$  and consider 
\begin{equation}\label{tmp:Q_def'}
	Q_1 := \{ (\vec{s}_*, t) ~:~  \sum_{s_1} f_{\vec{s}_*} (s_1) \overline{g_{\vec{s}_*} (ts_1)} \ge \sigma (2k \langle f,g \rangle _{k-1} )^{-1} \} \,.
\end{equation}
	Using the changing of the variables as in \eqref{f:comm_Uk} (here we appeal to the  commutativity of the group $\Gr$) and applying the argument as above, we have 
\[
	|Q_1| \cdot \sigma (2k \langle f,g \rangle _{k-1} )^{-1} \le \sum_{\vec{s}_*, t} \sum_{s_1} f_{\vec{s}_*} (s_1) \overline{g_{\vec{s}_*} (ts_1)} 
	=
	 \langle f,g \rangle _{k-1} \,.
\]
	Again  if the summation in (\ref{tmp:17.12_1}) is taken over the set $Q_1$, then it gives us $(1-1/2k)$ proportion of $\sigma$.
Hence by the standard  projection results see, e.g., \cite {Bol_Th} we see that the summation
in (\ref{tmp:17.12_1}) is taken over a set $\mathcal{S}$ of vectors $(\vec{s},t)$ of size at most
$$
	|\mathcal{S}| \le ( (2k)^{k+1} \| f \|_{\mathcal{U}^{k}} \| g \|_{\mathcal{U}^{k}} \sigma^{-(k+1)} )^{1/k} \langle f,g \rangle^2_{k-1} \,.
$$
	Whence by the Cauchy--Schwartz inequality, we get
\[
	2^{-4}\sigma^2 \le |\mathcal{S}| \sum_{\vec{s}, t} \left| \sum_{x} f_{\vec{s}} (x) \overline{g_{\vec{s}} (xt)} \right|^2 
		\le 
			( (2k)^{k+1}  \| f \|_{\mathcal{U}^{k}} \| g \|_{\mathcal{U}^{k}} \sigma^{-(k+1)} )^{1/k}  \langle f,g \rangle^2_{k-1}  \langle f, g\rangle^{}_{k+1}
\]
	and we have \eqref{f:Gowers_char_AB_1} up to a constant depending on $k$. 
	Using the tensor trick (e.g., see, \cite{TV}) we obtain the result with the constant one.

	To 
	prove inequality 
	\eqref{f:Gowers_char_AB_2} we see by induction and formula \eqref{f:Gowers_char_AB_1} that one has for $l\le k$ 
\[
	\langle A,B \rangle_l 
		\le 
			\langle A,B \rangle^{\a_0 (l,k)}_0 \prod_{j=1}^{k-1}  (\| A\|_{\mathcal{U}^j}  \| B\|_{\mathcal{U}^j})^{\a_j(l,k)} \cdot \langle A,B \rangle^{\beta_k(l,k)}_k \,,
\]
	where $\a_j (l,k)$, $\beta_j (l,k)$ are some non--negative functions.
	In principle, in view of \eqref{f:Gowers_char_AB_1} these functions can be calculated via some recurrences but we restrict ourselves giving  just crude bounds for them. 
	We are interested in 
	 $l=1$ and $k$ is a fixed number  and hence we write  
\[
\E(A,B) = 
	\langle A,B \rangle_1 
	\le 
\langle A,B \rangle^{\a_0}_0 \prod_{j=1}^{k-1}  (\| A\|_{\mathcal{U}^j}  \| B\|_{\mathcal{U}^j})^{\a_j} \cdot  \langle A,B \rangle^{\beta}_k \,.
\]	
	By homogeneity, we get
\begin{equation}\label{eq:homogeneity}
	2= \a_0 + \sum_{j=1}^{k-1} \a_j 2^j + 2^k \beta \,.
\end{equation}
	In particular, $\beta \le 2^{-k+1}$. 
	Further taking $A=B$ equals a subgroup, we obtain one more equation 
\begin{equation}\label{eq:energy}
	3= 2\a_0 + 2\sum_{j=1}^{k-1} \a_j (j+1) + (k+2) \beta \,.
\end{equation}
	Using trivial inequalities $\|A\|_{\mathcal{U}^j} \le |A|^{j+1}$, $\|B\|_{\mathcal{U}^j} \le |B|^{j+1}$  and formula \eqref{eq:energy}, we derive
\[
	\E(A,B) \le (|A| |B|)^{\a_0 + \sum_{j=1}^{k-1} \a_j (j+1)} \langle A,B \rangle^{\beta}_k 
	= (|A||B|)^{\frac{3}{2} - \frac{\beta(k+2)}{2}}  \langle A,B \rangle^{\beta}_k  
\]
	as required. 
	Actually, if there is a non-trivial upper bound for $\|A\|_{\mathcal{U}^j}$ (and it will be so in the next section), then the last estimate can be improved in view of Lemma \ref{l:Gowers_char}.
	Further our task is to obtain a good lower bound for $\beta$. 
	Put $\omega_j := 3+1/j > 3$, $j\in [k-1]$. 
	Using \eqref{f:Gowers_char_AB_1}, we get
\begin{equation}\label{f:prod_xj}
	\prod_{j=1}^{k-1} \langle A, B \rangle^{\o_j x_j}_j \le S \prod_{j=1}^{k-1} \left( \langle A, B \rangle^2_{j-1} \langle A, B \rangle_{j+1} \right)^{x_j} \,,
\end{equation} 
	where $S$ is  a quantity depending on $\|A\|_{\mathcal{U}^j}$, $\|B\|_{\mathcal{U}^j}$,  which we do not specify 
	and let $x_j$ be some positive  numbers, which we will choose (indirectly) later.
	For  $2 \le j \le k-2$ put 
	\begin{equation}\label{f:x_j}
		x_{j-1} + 2 x_{j+1} = \o_j x_{j}  \,.
	\end{equation}
	Then we obtain from \eqref{f:prod_xj} 
\[
	\langle A, B \rangle^{4x_1}_1 
		\langle A, B \rangle^{\omega_{k-1} x_{k-1}}_{k-1} 
			\le 
				S 
				\langle A, B \rangle^{2x_{2}}_1 \langle A, B \rangle^{x_{k-1}}_k \langle A, B \rangle^{x_{k-2}}_{k-1}	\,.			
\]
	Now choosing $x_{k-2} = \omega_{k-1} x_{k-1}$, we see that $\beta = x_{k-1}/(4x_1-2x_2)$ and it remains to estimate $x_{k-1}$ in terms of $x_1, x_2$. 
	But for  all $j$ one has $\o_j \le 4$, hence $x_{k-1} \ge 4^{-1} x_{k-2}$ and, similarly, from  
	$x_{j-1} + 2 x_{j+1} = \o_j x_{j}$,  $2 \le j \le k-2$, 
	we get 
	$x_{j} \ge 4^{-1} x_{j-1}$ and hence $x_{k-1} \ge 4^{-(k-1) }x_1$. 
	Further summing \eqref{f:x_j} over $2\le j \le k-2$ and putting $T=\sum_{j=1}^{k-1} x_j$, we obtain 
\[
	T-x_{k-1}-x_{k-2} + 2T - 2x_1 - 2x_2 = \sum_{j=2}^{k-2} \o_j x_j \ge  3T - 3x_1 -3x_{k-1}
\]
	and hence 
	$$
		x_1-2x_2 \ge x_{k-2} - 2x_{k-1} = (\o_{k-1} - 2) x_{k-1} > 0 \,.
	$$
	In particular, it gives $x_j >0$ for all $j\in [k-1]$ and thus indeed $\beta \ge 4^{-k}$.

	Now suppose that $\Gr$ is an arbitrary group and $\|A\|_{\mathcal{U}^k} \le |A|^{k+1-c}$ but $\E(A) \ge |A|^3/K$, where $K\ge 1$ is a parameter. 
	By the non--commutative  Balog--Szemer\'edi--Gowers Theorem, see \cite[Theorem 32]{Brendan_rich} or \cite[Proposition 2.43, Corollary 2.46]{TV}
	there is $a\in A$ and $A_* \subseteq a^{-1}A$, $|A_*|\gg_K |A|$ such that $|A^3_*| \ll_K |A_*|$.
	We can apply the previous argument to the set $A_*$ and obtain an estimate  similar to Lemma \ref{l:Gowers_char}
$$
	|A|^{k+1-c} \ge \| A \|_{\mathcal{U}^{k}} \ge \| A_* \|_{\mathcal{U}^{k}} \gg_K \E(A_*)^{2^{k-2}} |A_*|^{-(3\cdot 2^{k-2} -k -1)} 
	\gg_K 
$$
\begin{equation}\label{tmp:19.12_1} 
	\gg_K 
	\E(A)^{2^{k-2}} |A|^{-(3\cdot 2^{k-2} -k -1)}  \,.
\end{equation}
	Indeed, to bound  $\E(A)$ via $\| A\|_{\mathcal{U}^{k+2}}$ using the argument as in the proof above we need to estimate size of the set $\mathcal{S}_k$  at each step $k$. 
	But clearly $|\mathcal{S}_k| \le |AA^{-1}|^{k+1} \ll_K |A|^{k+1}$ and hence by induction we obtain $\E^{2^k}(A) \ll_K |A|^{3\cdot 2^k-k-3} \| A\|_{\mathcal{U}^{k+2}}$ as required.
	Finally, from \eqref{tmp:19.12_1}, it follows that $K^{C(k)}\gg |A|^{c}$, where $C(k)$ is a constant depending on $k$ only.
This completes the proof.
$\hfill\Box$
\end{proof}

\bigskip

A closer look to the proof (see, e.g., definition \eqref{tmp:Q_def'}) shows that 
for $k=1$  estimate \eqref{f:Gowers_char_AB_1} of  Lemma \ref{l:Gowers_char} takes place for any  class--functions $f$, $g$. 
Nevertheless,  
for larger $k$ this argument does not work.

\section{The proof of the main result}
\label{sec:proof}

Let $\Gr$ be an algebraic group in an affine or projective space of dimension $n$ over the  field $\F_q$, and let  $V\subseteq \Gr$ be a variety, $d=\dim(V)$, $D=\deg (V)$.  
If $V$ is absolutely irreducible, then by Lang--Weil  \cite{LW}  we know that
\begin{equation}\label{f:LW}
	\left| |V| - q^d \right| \le (d-1)(d-2) q^{d-1/2} + A(n,d,D) q^{d-1} \,,
\end{equation}
where $A(n,d,D)$ is a  certain constant. 
By sufficiently large $q$ we mean that $q\ge q_0 (n,d,D,$
$\dim(\Gr),\deg(\Gr))$ and all constants below are assumed to depend on $n,d,D,\dim(\Gr),\deg(\Gr)$.  
In particular, for an  absolutely irreducible variety $V$ one has $q^d \ll |V| \ll q^d$.
One can think about $\Gr$ and $V$ as varieties defined over $\mathbb{Q}$ by absolutely irreducible polynomials.
Then by the Noether Theorem \cite{Noether} we know that $\Gr$ and $V$ reduce 
$\mathrm{mod~} p$ 
to some absolutely irreducible varieties  defined over $p$, $p\notin S(\Gr,V)$, where $S(\Gr,V)$ is a certain finite set of the primes.

Finally, for any set $W \subseteq \Gr$ consider the quantity
\begin{equation}\label{def:t}
t = t(W) := \max_{x\in \Gr,\, \G \le \Gr}\, \{ |\G| ~:~ x\G \subseteq W,\,\, \G \mbox{ is an algebraic subgroup} \} \,.
\end{equation}

Now we are ready to estimate different energies of varieties in terms of the quantity $t(V)$. 
We are also able to give a non--trivial bound for any sufficiently large subset of $V$, see inequality \eqref{f:A_in_V_bound}. 

\begin{theorem}
	Let $\Gr$ be an algebraic group, $V\subseteq \Gr$ be
	a 
	variety, $d=\dim(V)$, $D=\deg (V)$, $t=t(V)$.   
	Then for any positive integer $k$ and all sufficiently large $q$ one has 
\begin{equation}\label{f:E_k_variety}
	\E_k (V) \ll_{d,D} \frac{|V|^{k+1}}{q^{k-1}} + t |V|^k \,.
\end{equation}
	In particular, if $V$ is absolutely irreducible and $V$ is not a coset of a subgroup, then 
	$\E_k (V) \ll_{d,D} |V|^{k+1 - \frac{1}{d}}$.\\   
	Similarly,  one has
\begin{equation}\label{f:U_k_variety}
	\| V\|_{\mathcal{U}^{k}} \ll_{d,D} |V|^{k+1} q^{-\frac{k(k-1)}{2}} + |V|^2 t^{k-1} + |V|^2 \sum_{j=1}^{k-2} |V|^{j} t^{k-1-j} q^{-\frac{j(1+j)}{2}} \,,
\end{equation}
	and for any $A\subseteq V$ the following holds 
\begin{equation}\label{f:A_in_V_bound}
	\| A\|_{\mathcal{U}^{d+1}} \ll_{d,D} t|A|^{d+1} \,.
\end{equation}
\label{t:E_k_variety}
\end{theorem}
\begin{proof} 
	First of all, consider the case of an absolutely irreducible $V$. 
	For any $g\in \Gr$ we have either $\dim(V \cap gV) < \dim (V)$, or $g$ belongs to the stabilizer $\Stab(V)$ of $V$. 
	It is well--known that any stabilizer under any action of an algebraic group is an algebraic subgroup (but not necessary irreducible). 
	Clearly, $\Stab(V) \subseteq v^{-1}V$ for any $v\in V$ and hence either $V$ is a coset of an (algebraic) subgroup, or 
	$\dim (\Stab(V)) < \dim (V)$.
	The degree of the variety $gV \cap V$ is at most $\deg^2 (V)$ by inequality \eqref{f:deg_intersection}.
	Similarly, since our topological space is a Noetherian one (see, e.g., \cite[page 5]{Hartshorne}) and $\Stab(V) = \bigcap_{v\in V} v^{-1}V$, it follows that cardinality of $\Stab(V)$ can be estimated in terms of $d$ and $D$ thanks to \eqref{f:LW}. 
	Hence all parameters of all appeared varieties are controlled by $d,D$, $n$ (and, possibly, by $\dim(\Gr), \deg(\Gr))$.   
	Using Lang--Weil  formula \eqref{f:LW}, we obtain for sufficiently large $q$ that  $q^d/2 \le |V| \le 2q^d$, say, further $|\Stab(V)| \ll q^{d-1}$, and, similarly, 
	for any $g\notin \Stab(V)$ one has $|V \cap gV| \ll q^{d-1}$. 
	Hence
\begin{equation}\label{f:irr_1}
	\E_k (V) = \sum_{g\in \Gr} |V\cap gV|^k = \sum_{g\notin \Stab(V)} |V\cap gV|^k + \sum_{g\in \Stab(V)} |V\cap gV|^k 
		\ll 
\end{equation}
\begin{equation}\label{f:irr_2}
		\ll 
		(q^{d-1})^{k-1} \sum_{g\in \Gr} |V\cap gV| + q^{d-1} |V|^k \ll (q^{d-1})^{k-1} |V|^2 + q^{d-1} |V|^k
		\ll
		|V|^{k+1 - \frac{1}{d}} \,.
\end{equation}

	 Now to obtain \eqref{f:E_k_variety} we apply the same argument but before we need to consider $V$ as a union of its irreducible components $V=\bigcup_{j=1}^s V_j$.
	 Clearly, $s\le \deg (V)$. 
	 Take any $g\in \Gr$ and consider $V\cap gV = \bigcup_{i,j=1}^s (V_i \cap g V_j)$.
	 If for all $i,j\in [s]$ one has $\dim(V_i \cap g V_j) < \dim V$, then for such $g$ we can repeat the previous calculations in \eqref{f:irr_1}--\eqref{f:irr_2}.
	 Consider the set of the remaining $g$ and denote this set by $B$. 
	 For any $g\in B$  there is $i,j\in [s]$ such that $\dim(V_i \cap g V_j) = \dim(V)$.
	 In particular, $\dim(V_i) = \dim(V_j) = \dim(V)$ and $V_i \cap g V_j = V_i = gV_j$ by irreducibility of $V_i,V_j$. 
	 Suppose that for the same pair $(i,j)$ there is another $g_* = g_* (i,j) \in B$ such that $g_* V_j = V_i$.
	 Then $g^{-1}_* g \in \Stab(V_j)$. 
	 It follows that $g\in g_* \Stab(V_j)$ and hence the set $B$ belongs to $\bigcup_{i,j=1}^s g_* (i,j) \Stab(V_j)$ plus at most $s^2 \le \deg^2 (V)$ points. 
	Hence we need to add to  the computations in \eqref{f:irr_1}--\eqref{f:irr_2} the term
\[
	s^2 (t+1) |V|^k  \le \deg(V)^2 (t+1) |V|^k \ll t |V|^k
\]
	as required.

	To prove 
	bound \eqref{f:U_k_variety} 
	let us obtain a generalization of \eqref{f:E_k_variety}. 
	Put $\E^{(l)}_k = \E^{(l)}_k (V) := \sum_{\vec{s}} |V_{\vec{s}}|^k$, where $\vec{s} = (s_1,\dots,s_l)$ and $V_{\vec{s}}$ as in \eqref{def:A_vec{s}}.
	Write $\vec{s} = (\vec{s}_*,s_l)$.
	Since $\E^{(l)}_k = \sum_{\vec{s}_*} \E_k (V_{\vec{s}_*})$, it follows that by the obtained estimate \eqref{f:E_k_variety}
\begin{equation}\label{f:E_kl-}
	\E^{(l)}_k \ll \sum_{\vec{s}_*} \left( \frac{|V_{\vec{s}_*}|^{k+1}}{q^{k-1}} + t (V_{\vec{s}_*}) |V_{\vec{s}_*}|^{k} \right)
		 \ll q^{-(k-1)} \E^{(l-1)}_{k+1} + t \E^{(l-1)}_{k} \,.
\end{equation}
	Here we have used the fact that the function $t$ on a subset of $V$ does not exceed $t(V)$.
	Further inequality \eqref{f:E_kl-} gives  by induction 
\begin{equation}\label{f:E_kl}
	\E^{(l)}_k \ll |V|^k \sum_{j=0}^l q^{-\frac{j(2k+j-3)}{2}} |V|^j t^{l-j} \,.
\end{equation}
	Now, applying inequality \eqref{f:E_k_variety} with the parameter $k=2$ and using the notation $\vec{s} = (\vec{s}_*,s_l)$ again, we get 
\[
	\| V\|_{\mathcal{U}^{l+1}} = \sum_{\vec{s}} |V_{\vec{s}}|^2 = \sum_{\vec{s}_*} \E (V_{\vec{s}_*}) 
		\ll 
			\sum_{\vec{s}_*} \left( \frac{|V_{\vec{s}_*}|^3}{q} + t |V_{\vec{s}_*}|^2 \right)
	=
	q^{-1} \E^{(l-1)}_3 + t \E^{(l-1)}_2.
\]
	Hence in view of \eqref{f:E_kl}, we derive
\[
	\| V\|_{\mathcal{U}^{l+1}} \ll |V|^2 \sum_{j=0}^{l-1} |V|^{j} t^{l-1-j} (|V| q^{-\frac{j(3+j)}{2}-1} + tq^{-\frac{j(1+j)}{2}}) 
	\ll
\]
\[
	\ll
	|V|^2 t^l + |V|^{l+2} q^{-\frac{l^2+l}{2}} + |V|^2 \sum_{j=1}^{l-1} |V|^{j} t^{l-j} q^{-\frac{j(1+j)}{2}} \,.
\]
	
	Finally, take any $A\subseteq V$.
	As above put  $\vec{s} = (s_1,\dots,s_{d}) = (\vec{s}_*, s_d)$.
	We have $A_{\vec{s}} \subseteq V_{\vec{s}}$.
	Further by \eqref{f:A,B_scalar_exm} 
\begin{equation}\label{tmp:02.12_1} 
	\| A \|_{\mathcal{U}^{d+1}} = \sum_{\vec{s}} |A_{\vec{s}}|^2 = \sum_{\vec{s}_*} \E (A_{\vec{s}_*}) = \sum_{\vec{s}_*} \sum_{s_d} |A_{\vec{s}_*} \cap A_{\vec{s}_*} s_d|^2 \,.
\end{equation}
	Take any vector $\vec{z} = (z_1,\dots, z_{l})$, $l< d$ and consider  the decomposition of the variable $V_{\vec{z}}$  onto irreducible components 
	$V_{\vec{z}} (j)$.
	As above define the set $B(\vec{z})$  of all $g\in \Gr$ such that there are $V_{\vec{z}} (i), V_{\vec{z}} (j)$ with $V_{\vec{z}} (i) =  gV_{\vec{z}} (j)$.
	Then by the arguments as before, we have $|B(\vec{z})| \ll t$. 
	Using formula \eqref{tmp:02.12_1}, we get 
\begin{equation}\label{tmp:very_end-}
	\| A \|_{\mathcal{U}^{d+1}} \ll t \sum_{\vec{s}_*} |A_{\vec{s}_*}|^2  + \sigma = t \| A \|_{\mathcal{U}^{d}} + \sigma \,,
\end{equation}
	where for all $\vec{s}$ in $\sigma$, we have $\dim (V_{\vec{s}}) = 0$.
	Hence
\begin{equation}\label{tmp:very_end}
	\| A \|_{\mathcal{U}^{d+1}} \ll t \| A \|_{\mathcal{U}^{d}} + \sum_{\vec{s}} |A_{\vec{s}}| \ll t |A|^{d+1} + |A|^{d+1} \ll t |A|^{d+1} \,. 
\end{equation}	
This completes the proof.
$\hfill\Box$
\end{proof}

\bigskip

Once again the bounds above depend on $d,D$, as well as on $\dim(\Gr)$, $\deg(\Gr)$ and on dimension of the ground affine space.

\begin{remark}
	The quantity $\| V\|_{\mathcal{U}^{k}}$ can be written in different ways as $\sum_{s_1,\dots,s_k} |V_{s_1,\dots,s_k}|^2$,\\  
	$\sum_{s_1,\dots,s_{k-1}} \E (V_{s_1,\dots,s_{k-1}})$ and so on. 
	Taking variables $s_j$ running over  the maximal coset belonging to $V$ we see that all  terms with $t$ in \eqref{f:U_k_variety} are needed.    
\end{remark}

\begin{corollary}
	Let $\Gr$ be an abelian  algebraic group, $V\subseteq \Gr$ be 
	a 
	variety, $d=\dim(V)$, $D=\deg (V)$.  
	Then for all sufficiently large $q$ and any $A\subseteq V$ one has 
\begin{equation}\label{f:E(A)_in_V}
	\E(A) \ll_{d,D}  |A|^3 \left(\frac{t}{|A|} \right)^{(2^{d+1}-d-5)^{-1}} \mbox{ for   } d\ge 2 \quad \mbox{and} \quad \E(A) \ll_{d,D}  |A|^2 t \,,\mbox{   for } d=1 \,.
\end{equation}
	In particular, for any $A\subseteq V$ with $|A| \ge t^{1+c}$, $c>0$ there is $\delta=\delta(d,c)>0$ such that 
\begin{equation}\label{f:E(A)_3-d} 
	\E(A) \ll_{d,D} |A|^{3-\delta} \,.
\end{equation}
	Bound \eqref{f:E(A)_3-d} takes place in any algebraic group.\\ 
 	Moreover, let $B\subseteq \Gr$ be an arbitrary set.
 	Then  
 \begin{equation}\label{f:E(A,B)_in_V}
 	\E(A,B) \ll_{d,D}  \left( \frac{t}{|A|} \right)^{\beta} \cdot  |A|^{\frac{3}{2} - \frac{\beta d}{2}} |B|^{\frac{3}{2} + \frac{\beta d}{2}} \,,
 \end{equation}
 	where $\beta = \beta(d) \in [4^{-d}, 2^{-d+1}]$, 
	and for any $k\ge 1$ one has either $\T_k (A) \le |A|^{2k-1-c\beta/4}$ or 
\begin{equation}\label{f:E(A)_in_V2}
	\T_{k+1} (A) \ll_{d,D}  |A|^{2}  \T_k (A) \cdot |A|^{-c\beta/4} \,. 
\end{equation}
\label{cor:E(A)_in_V}
\end{corollary}
\begin{proof} 
	Let $d\ge 2$. 
	By Theorem \ref{t:E_k_variety}, we have $\| A\|_{\mathcal{U}^{d+1}} \ll t |A|^{d+1}$. 
	Using the second part of Lemma \ref{l:Gowers_char} with $k=d+1$, we obtain 
\[
	\E(A) \ll |A|^3 \left(\frac{t}{|A|} \right)^{(2^k-k-4)^{-1}} =  |A|^3 \left(\frac{t}{|A|} \right)^{(2^{d+1}-d-5)^{-1}} \,.
\]
	If $d=1$, then the arguments of the proof of Theorem \ref{t:E_k_variety} (see, e.g., \eqref{tmp:very_end}) give us 
\[
	\E(A) \ll t |A|^2 + \sum_s |A_s| \ll t |A|^2 \,.
\]
	For an arbitrary  algebraic group use the last part of Lemma \ref{l:Gowers_char_AB}.

	To derive \eqref{f:E(A,B)_in_V} we can suppose that $|B| \ge |A|$ because otherwise the required bound 
	$$
		\E(A,B)^2 \le \E(A) \E(B) \le \left( \frac{t}{|A|} \right)^{\beta}  |A|^{3} |B|^3 \le  \left( \frac{t}{|A|} \right)^{\beta} \cdot  |A|^{3 - \beta d} |B|^{3 + \beta d} 
	$$
	takes place for $\beta =   (2^{d+1}-d-5)^{-1}$, $d\ge 2$ (and similar for $d=1$),  see estimate \eqref{f:E(A)_in_V}.  
	Further let $|B|\ge |A|$.
	Then we use Lemma \ref{l:Gowers_char_AB} with $k=d$, combining with Theorem \ref{t:E_k_variety} (see formulae \eqref{f:A,B_scalar_exm}, \eqref{tmp:very_end-},  \eqref{tmp:very_end}) and the assumption $|A| \le |B|$ to obtain 
	\[
	\E(A,B) \le (|A||B|)^{\frac{3}{2} - \frac{\beta(d+2)}{2}}  \langle A,B \rangle^{\beta}_d  
	\ll
	(|A||B|)^{\frac{3}{2} - \frac{\beta(d+2)}{2}}  \left( \| B \|_{\mathcal{U}^d} + t \langle A,B \rangle_{d-1} \right)^{\beta} 
	\le
\]
\[
	\le 
	(|A||B|)^{\frac{3}{2} - \frac{\beta(d+2)}{2}}  \left( |B|^{d+1} + t |A| |B|^d \right)^{\beta} 
	\ll 
	(|A||B|)^{\frac{3}{2} - \frac{\beta(d+2)}{2}}  t^{\beta} |B|^{\beta (d+1)} 
	=  t^{\beta}  |A|^{\frac{3}{2} - \frac{\beta(d+2)}{2}} |B|^{\frac{3}{2} + \frac{\beta d}{2}}  
\]
	and \eqref{f:E(A,B)_in_V} follows. 
	Finally, to get \eqref{f:E(A)_in_V2} we use the dyadic pigeon--hole principle and the fact that $\T^{1/2k} (f)$  defines a norm of $f$ 
	to find the number $\D>0$ and the set $P$ such that   
	$P = \{ x\in \Gr ~:~ \D < A^{(k)} (x) \le 2\D \}$ and $\T_{k+1} (A) \lesssim \D^2 \E(A,P)$.
	Thus (we assume that $c\le 1$)
\[
	\T_{k+1} (A) \lesssim   (t/|A|)^{\beta} |A|^{\frac{3}{2} - \frac{\beta d}{2}} (\D^2 |P|)^{\frac{1}{2}-\frac{d\beta}{2}} (\D |P|)^{1+d\beta}
	\le 
|A|^{-\frac{c\beta}{2}}
	|A|^{\frac{3+2k}{2} - \frac{\beta d}{2}+\beta kd} \cdot \T^{\frac{1}{2}-\frac{d\beta}{2}}_k (A) \,.
\]
	Suppose that $\T_k (A) \ge |A|^{2k-1-\eps}$, where $\eps \le c\beta/4$. 
	In view of the last inequality and $\beta \le 2^{-d+1}$ one has 
$$
	|A|^{-\frac{c\beta}{2}} |A|^{\frac{3+2k}{2} - \frac{\beta d }{2}+\beta kd} \cdot \T^{\frac{1}{2}-\frac{d\beta}{2}}_k (A) 
	\le |A|^{2-\eps} \T_k (A) \,.
$$
This completes the proof.
$\hfill\Box$
\end{proof}


\begin{remark}
	As it was said in the proof of Lemma \ref{l:Gowers_char_AB} the bound for $\E(A,B)$, $A\subseteq V$, where $V$ is our variety and $B\subseteq \Gr$ is an arbitrary set can be improved because we have a non--trivial upper bound for $\| A\|_{\mathcal{U}^l}$, $2\le l\le d+1$.
	Thus bounds \eqref{f:E(A,B)_in_V}, \eqref{f:E(A)_in_V2} can be improved slightly.\\
	Also, inequalities  \eqref{f:E(A,B)_in_V}, \eqref{f:E(A)_in_V2} say, basically,  that either  $|A^3|$ is much larger than $|A|$ or $|A^3|$ is larger than $|A^2|$.  
\end{remark}


Theorem \ref{t:E_k_variety} and Corollary \ref{cor:E(A)_in_V} imply the following criterion.  

\begin{corollary}
	Let $\Gr$ be a finite simple group, $V\subseteq \Gr$ be a variety, $d=\dim(V)$, $D=\deg (V)$.  
	Suppose that $t(V) = o(|V|)$.
	Then there is $\d=\d(d,n)>0$ such that for any $A\subseteq V$, $|A| \gg |V|$ the following holds  
\begin{equation}\label{f:FF(V)}
	\E(A) \ll_{d,D} |A|^{3-\delta} \,.
\end{equation}
\label{c:FF(V)}
\end{corollary}

Clearly, if $t(V) \gg |V|$, then \eqref{f:FF(V)} does not hold and hence Corollary \ref{c:FF(V)} is indeed a criterion.
Also, it gives a lower bound for $\delta$ of the form $\delta \gg 1/d$. 
Recall that our current dependence on $d$ in \eqref{f:FF(V)} has an exponential nature.

\section{Applications}
\label{sec:applications}


In \cite{LS1}, \cite{LSS} authors obtain the following results on growth of normal sets.

\begin{theorem}
	Let $\Gr$ be a finite simple group and $N\subseteq \Gr$ be a normal set.
	Then there is $n\ll \log |\G|/\log |N|$ such that $N^n = \Gr$. 
	Moreover, for any $\eps>0$ there is $\d=\d(\eps)>0$ such that any normal set $N$ with $|N| \le |\Gr|^\d$ satisfies $|NN| \ge |N|^{2-\eps}$. 
\label{t:gr_conj}
\end{theorem}

From Corollary \ref{cor:E(A)_in_V} we obtain a result on growth of an {\it arbitrary} subset of a conjugate class.

\begin{corollary}
	Let $\Gr$ be a finite connected semisimple algebraic group and let $C \subseteq \Gr$ be a  conjugate class.
	Also, let $A\subseteq C$ be an arbitrary set with $|A| \ge t(\mathrm{Zcl}(C))^{1+\eps}$.   
	Then for a certain $\d>0$ depending on dimension of $C$ one has  $\E(A) \ll_{\deg(\mathrm{Zcl}(C))} |A|^{3-\d}$.
	In particular, $|AA| \gg_{\deg(\mathrm{Zcl}(C))} |A|^{1+\d}$. 
\label{c:bounded_rank}
\end{corollary}
\begin{proof} 
	In is well--known (e.g., see, \cite[pages 15, 17]{Humphreys})  that for any conjugate class $C$ its Zariski closure $\mathrm{Zcl}(C)$ equals $C$ and possibly other conjugate classes of strictly lower dimension, as well as that $C=C(x)$ is a variety iff  $x \in \Gr$ is a semisimple element.
	Now the result follows from a direct application of Corollary \ref{cor:E(A)_in_V} where the implied constants depend on $\deg (\mathrm{Zcl}(C))$, $\dim (C)$,   $\dim(\Gr), \deg(\Gr)$ and dimension of the ground affine space.       
This completes the proof.
$\hfill\Box$
\end{proof}

\bigskip

One can see that $t(C) \le |C|^{1-c_*}$ for a certain $c_*>0$  
via the general bound on such intersections with generating sets, see \cite{LP} or, alternatively, from some modifications of Theorem \ref{t:gr_conj}, see \cite{LS1}, \cite{LSS}.   
Thus Corollary \ref{c:bounded_rank} takes place for all large subsets of conjugate classes.

\bigskip

{\bf Question.} 
Is it true that for any $A\subseteq C$, where $C$ is a conjugate class such that  $|A|\ge |C|^{1-o(1)}$, say, one has $A^n = \Gr$, where $n$ is a function on
$\log |\Gr|/\log |A|$? For  $n\ll \log |\Gr|/\log |A|$?
For $C=C(x)$, where $x$ is a semisimple element?

\bigskip 

Now we are ready to obtain a non--trivial upper bound for any sufficiently large subset of a  Chevalley group living in a variety differ from the maximal parabolic subgroup.

\begin{theorem}
	Let $\Gr_r (\F_q)$ be a finite Chevalley group with rank $r$ and odd $q$  and $\Pi \le \Gr_r (\F_q)$ be its a maximal (by size) parabolic subgroup.
	Also, let $V \subset \Gr_r (\F_q)$ be a variety differ from all shifts of conjugates of $\Pi$.
	Then for any $A\subseteq V$, $|A| \ge |\Pi| q^{-1+c}$, $c>0$ one has 
\begin{equation}\label{f:Chevalley_A}
	\| \FF{A} (\rho) \|_o \le |A|^{1-\d} \,,
\end{equation}
	where $\d = \d(c,r) >0$ and $\rho$ is any non--trivial representation of $\Gr_r (\F_q)$. 
\label{t:Chevalley_A}
\end{theorem}
\begin{proof} 
	Since by the assumption $|A| \ge |\Pi| q^{-1+c}$, it follows that $|V| \ge |\Pi| q^{-1+c}$ and hence $V$ is rather large.
	Also, 
	one can see 
	that, trivially,  $|A|\ge |\Pi| q^{-1+c} \gg q^{1+c}$. 
	Further by \cite[Lemma 8]{s_Chevalley} we know that $\Pi$ is the maximal  (by size) subgroup of $\Gr(\F_q)$ and for all other  subgroups $\G \le \Gr(\F_q)$ one has $|\G| \le q^{-1} |\Pi|$ ($\G$ is not conjugate to $\Pi$, of course).   
	In particular, in view of \eqref{f:LW} 
	$$
		t(V) \le \max_{x,y \in \Gr_r (\F_q)} \{ q^{-1} |\Pi|, |V \cap x\Pi y| \} \ll |V| q^{-c_*} \,,
	$$ 
	where $c_* = \min \{c,1\}$. 
	Take any subgroup $H$ differ from all conjugates of $\Pi$.
	Also, let  $x\in \Gr_r (\F_q)$ be  an arbitrary element.
	Our task is to estimate above size of the intersection $A_* := A\cap xH \subseteq V$.
	By Corollary \ref{c:FF(V)} and estimate \eqref{f:energy_CS}, we have 
$$
	|A_*|^{1+\d} \ll |A^{-1}_* A_*| \le |H| \le |\Pi| q^{-1} \le |A| q^{-c}
$$  
	and hence in particular, $|A_*| \ll |A|^{(1+\d)^{-1}}$ (actually, in this place of the proof we can assume a weaker condition on size of $A$). 
	By a similar argument and estimate \eqref{f:LW}, we derive that 
$$
	|A\cap x \Pi y| \le |V\cap x\Pi y| \ll |\Pi|q^{-1} \le |A| q^{-c}  
$$
	and hence for {\it any} proper subgroup $\G \subset \Gr_r (\F_q)$ and for all $x\in \Gr_r (\F_q)$  one has $|A\cap x\G| \ll |A|^{} q^{-c/2}$ (we use $|A| \gg q$ and assume that $\d\le c$). 
	In particular, $A$ is a generating set of $\Gr_r (\F_q)$. 
	Combining this observation with the fact (see \cite{LS_representations}) that Chevalley groups are quasi--random in the sense of Gowers \cite{Gow_random}, we obtain desired estimate \eqref{f:Chevalley_A}, see, e.g., \cite{H}, \cite{H_ideas} and  \cite[Sections 8,10]{sh_as}. 
This completes the proof.
$\hfill\Box$
\end{proof}

\bigskip

Now we obtain an application of  Corollary \ref{cor:E(A)_in_V} to some questions about the restriction phenomenon.
In this setting our group $\Gr$ is $\Gr = \F^n$, $\F$ is  a finite field, $V\subseteq \F^n$ is a variety  and $\Gr$ acts on $\Gr$ via shifts.  
For any function $g:\F^n \to \C$ consider  the commutative analogue of \eqref{f:Fourier_representations}  
$$
\hat g(\xi) := \sum_{x\in \F^n} g(x) e(-x\cdot \xi)\,, 
$$
as well as the inverse Fourier transform of a function $f:V \to \C$
$$
	(fd\sigma)^\lor(x) := \frac{1}{|V|} \sum_{\xi\in V} f(\xi) e(x\cdot\xi)\,,
$$
where $e(x\cdot\xi) = e^{2\pi i (x_1\xi_1 + \dots +x_n \xi_n)}$ for $x = (x_1,\dots,x_n)$, $\xi = (\xi_1,\dots,\xi_n)$. 
Thus a "Lebesgue $L^q$-norm" of $f$ on $V$ is defined as
$$
\|f\|_{L^q(V,d\sigma)}:= \left( \frac{1}{|V|} \sum_{\xi\in V} |f(\xi) |^q \right)^{\frac{1}{q}}\,,
$$
while for a function $g$ 
it is 
$$
\|g\|_{L^q(\F^n)}:= \left( \sum_{x\in \F^n} |g(x) |^q \right)^{\frac{1}{q}}\,.
$$
The finite field restriction problem \cite{MoT} for our variety $V$ seeks exponents pairs $(q,r)$ such that one has the inequality
$$
	\left\|(f d \sigma)^{\vee}\right\|_{L^{r}\left(\F^{n}\right)} \leq R^{*}(q \rightarrow r)\|f\|_{L^{q}(V, d \sigma)} 
$$
	or, equivalently, 
$$
	\|\widehat{g}\|_{L^{q^{\prime}}(V, d \sigma)} \leq R^{*}(q \rightarrow r)\|g\|_{L^{r^{\prime}}\left(\F^{n}\right)}
$$
	takes place with a constant $R^{*}(q \rightarrow r)$ independent of the size of the finite field.
	As before we use the notation  $\lesssim$ and $\gtrsim$ instead of $\ll$, $\gg$ allowing ourselves to lose logarithmic powers of $|\F|$.

\bigskip

Using the arguments of the proofs of \cite[Lemma 5.1, Proposition 5.2]{MoT}, we obtain

\begin{theorem}
	Let $V\subseteq \F^n$ be a variety, $d=\dim(V)$.
	Suppose that $V$ does not contain any line. 
	Then $R^* (\frac{4}{3-c} \rightarrow 4) \lesssim 1$, where $c=c(d) > 0$. 
\label{t:restriction}
\end{theorem} 
\begin{proof} 
	According our assumption that $V$ does  not contain any line, we see that the parameter $t(V)$ equals $1$. 
	Hence by Corollary \ref{cor:E(A)_in_V} we know that $\E(A) \ll |A|^{3-c} = |A|^\kappa$, where $c=c(d)>0$.  
	Put $q=4/\kappa$ and we want to obtain a good   bound for  $R^* (q \rightarrow 4)$. 
	We want to obtain an estimate of the form  (see the  proofs of \cite[Lemma 5.1, Proposition 5.2]{MoT})
\[
	\sum_x (fV * fV )^2 (x) \lesssim  \left( \sum_{x\in V} |f(x)|^q \right)^{4/q} \,,
\]
	where $f$ is an arbitrary  function (we can freely assume that $f$ is positive). 
	Using the dyadic pigeon--hole principle, we need to prove the last bound for any $f=A$ with $A\subseteq V$ and this is equivalent to  
\[
	\E(A) \lesssim |A|^{4/q} = |A|^{3-c} \,.
\]
This completes the proof.
$\hfill\Box$
\end{proof}

\bigskip

Notice that if the variety $V$ contains subspaces of positive dimension, then there is no any restriction--type result 
as in 
Theorem \ref{t:restriction} in such generality see, e.g., \cite[Section 4]{MoT}.

\section{Appendix}
\label{sec:appendix}


Now we obtain an analogue of the Weyl criterion for non--commutative 
case. 
In this situation ordinary abelian intervals or progressions correspond to some structural non--abelian objects as subgroups. 
In particular, the first part of proposition below is applicable for subgroups $H$ of our group $\Gr$. 
Of course such results should be known but it is difficult to find them in the literature and we include Proposition \ref{p:UD_non-commutative}  and its converse for the  completeness.

\begin{proposition}
	Let $\eps \in (0,1]$ be a real number,  $\Gr$ be a finite group, $A\subseteq \Gr$ be a set such that 
	for any non--trivial irreducible representation $\rho$ one has 
	\begin{equation}\label{l:UD_1-delta}
	\| \FF{A} (\rho) \|_{o} \le \eps |A| \,.
	\end{equation}  
	Then for any $H, H_*\subseteq \Gr$, $1\in H_*$ with $|HH_*| \le |H| + K|H_*|$ one has 
	\begin{equation}\label{f:UD_non-commutative} 
	\left| |A\cap H| - \frac{|A||H|}{|\Gr|} \right| \le 2 K|H_*| + \eps |A| \sqrt{|H|/|H_*| + K} \,.
	\end{equation}
\label{p:UD_non-commutative} 
\end{proposition}
\begin{proof} 
	Put $\Pi = HH_*$. Then for any $x\in H$ the following holds $H(x) = |H_*|^{-1} (\Pi * H_*^{-1} ) (x)$.
	Hence
	\[
	\| H(x) - |H_*|^{-1} (\Pi * H_*^{-1} ) (x) \|_1 \le |HH_*| - |H|  \le K |H_*| \,, 
	\]
	and thus in view of formulae \eqref{f:Parseval_representations}, \eqref{f:convolution_representations}, we obtain 
	\[
	|A\cap H| = |H_*|^{-1} \sum_x A(x) (\Pi * H_*^{-1} ) (x) + \mathcal{E} 
	= 
	\frac{|A||\Pi|}{|\Gr|} + \frac{1}{|H_*||\Gr|} \sum_{\rho \in \FF{\Gr},\, \rho \neq 1} d_\rho \langle \FF{A} (\rho),  \FF{\Pi} (\rho)  \FF{H}^*_* (\rho) \rangle    + \mathcal{E} \,,
	\]
	where $|\mathcal{E}| \le K|H_*|$. 
	Applying condition \eqref{l:UD_1-delta}, the Cauchy--Schwarz inequality and formula \eqref{f:Parseval_representations} again,  we get 
	\[
	\left| |A\cap H| - \frac{|A||H|}{|\Gr|} \right| \le K|H_*| + \frac{K|A||H_*|}{|\Gr|} 
	+ \eps |A| |\Pi|^{1/2} |H_*|^{-1/2} \le 2 K|H_*| + \eps |A| \sqrt{|H|/|H_*| + K} \,.
	\]
	This completes the proof.
	$\hfill\Box$
\end{proof}

\bigskip

The inverse statement to Proposition \ref{p:UD_non-commutative} also takes place but it requires some notation 
and, actually, our argument gives an effective bound if dimension of the correspondent representation $\rho$ is small. 
Following \cite[Section 17]{Sanders_A(G)} define the Bohr sets in a (non--abelian) group $\Gr$.

\begin{definition}
	Let $\G$ be a collection of some unitary representations of $\Gr$ and $\delta \in (0,2]$ be a real number.
	Put 
	\[
	\Bohr (\G,\delta) = \{ g\in \Gr ~:~ \| \gamma(g) - I \|_o \le \delta\,, \forall \gamma \in \Gamma  \} \,.
	\] 
	The number $|\G|$ is called the {\it dimension} of $\Bohr (\G,\delta)$. 
	If $\G = \{ \rho \}$, 
	then we write just $\Bohr (\rho,\delta)$ for $\Bohr (\{ \rho \},\delta)$. 
	A Bohr set $\Bohr (\rho, \delta)$ is called to be {\it regular} if 
	\[
	\left| |\Bohr (\rho, (1+\kappa)\delta)| -  |\Bohr (\rho, \delta)| \right|_o \le 100 d^2_\rho |\kappa| \cdot |\Bohr (\rho, \delta)| \,, 
	\] 
	whenever $|\kappa| \le 1/(100 d^2_\rho)$. 
\end{definition}

Even in the abelian case it is easy to see  that not each Bohr set is regular (e.g., see, \cite[Section 4.4]{TV}). 
Nevertheless, it can be showed 
(e.g., see, \cite{s_Laplace})  that  
one can find a regular Bohr set decreasing the parameter $\delta$ slightly.

\begin{lemma}
	Let $\delta \in [0,1/2]$ be a real number and $\rho$ be a unitary representation. 
	Then there is $\delta_1 \in [\delta,2\delta]$ such that $\Bohr (\rho, \delta_1)$ is regular. 
\end{lemma}

Let us remark an universal lower bound for size of any Bohr set (see \cite[Lemma 17.3]{Sanders_A(G)} and \cite[Proposition 28]{s_Laplace} for the case of multi--dimensional Bohr sets). 

\begin{lemma}
	Let $\delta \in (0,2]$ be a real number and $\Bohr (\rho, \delta) \subseteq \Gr$ be  a one--dimensional Bohr set. 
	Then 
\[
	|\Bohr (\rho, \delta)| \ge (c\delta)^{d^2_\rho} \cdot |\Gr| \,,
\]
	where $c>0$ is an absolute constant. 
\label{l:Bohr_size_nc}
\end{lemma}

Now suppose that for a set $A\subseteq \Gr$ one has $|A| = \d |\Gr|$ and $\| \FF{A} (\rho) \|_o \ge \eps |A|$.  
Put $f (x) = f_A (x) = A(x) - \d$.
Take a regular Bohr set $B= \Bohr (\rho, \delta)$, $\delta = \eps/4$ and let $B_* =  \Bohr (\rho, \kappa \delta)$, where $|\kappa| \le 1/(100 d^2_\rho)$
is a certain number.  
Then by the definition of Bohr sets, we have 
\[
	\eps |A| \le \| \FF{A} (\rho) \|_o = |B|^{-1} \| \sum_h \sum_g f(g) B(gh^{-1}) \rho(g) \|_o = |B|^{-1} \| \sum_h \sum_g f(g) B(gh^{-1}) \rho(h) \|_o + \mathcal{E} \,,  
\]
	where $|\mathcal{E}| \le 2\d |A|$. 
	Thus
\[
	\eps |A|/2 \le |B|^{-1} \sum_h \left| \sum_g f(g) B(gh^{-1}) \right|  
\]
	and hence 
	in view of 
	Lemma \ref{l:Bohr_size_nc}, we find $h\in \Gr$ with  
\[
	\frac{|A||B|}{|\Gr|} + \eps |A| \exp (-O(d^2_\rho \log (1/\d))) \le  |A\cap Bh| \,.
\]
	On the other hand, by the regularity of $B$ one has $|BB_*| \le |B| (1+ 100 d^2_\rho |\kappa|)$. 
	It implies that Proposition \ref{p:UD_non-commutative} can be reversed indeed.

\bigskip

\noindent{I.D.~Shkredov\\
	Steklov Mathematical Institute,\\
	ul. Gubkina, 8, Moscow, Russia, 119991}
\\
and
\\
IITP RAS,  \\
Bolshoy Karetny per. 19, Moscow, Russia, 127994\\
and 
\\
MIPT, \\ 
Institutskii per. 9, Dolgoprudnii, Russia, 141701\\
{\tt ilya.shkredov@gmail.com}

\end{document}